\documentclass[12pt,a4paper]{article}
\usepackage{amsmath,amssymb,amsfonts}
\usepackage{amsthm}
\usepackage{mathtools}
\usepackage{marvosym}
\usepackage{authblk}
\usepackage{enumerate}
\usepackage[english]{babel}
\usepackage{enumitem}
\usepackage{todonotes}
\usepackage{indentfirst}
\usepackage{geometry}
\usepackage{graphicx}
\usepackage{tabularx}
\usepackage{tocbibind}
\usepackage{tikz}
\usetikzlibrary{graphs, positioning}

\usepackage{array}
\usepackage{caption}
\usepackage{setspace}
\usepackage{hyperref}
\usepackage{lipsum}
\usepackage{listings}
\usepackage{multicol,multirow}
\usepackage{xcolor,xspace}
\usepackage[english]{babel}
\usepackage[export]{adjustbox}

\newtheorem{theorem}{Theorem}[section]

\newtheorem{prop}[theorem]{Proposition}

\theoremstyle{definition}
\newtheorem{definition}{Definition}[section]
\theoremstyle{remark}

\title{\textbf{Bounds on Elliptic Sombor and Euler Sombor indices of join and corona product of graphs}}
\author[1]{Bishal Sonar\thanks{Email: bsonarnits@gmail.com}}
\author[2]{Ravi Srivastava\thanks{Corresponding author Email: ravi@nitsikkim.ac.in}}
\affil[1,2]{Department of Mathematics, National Institute of Technology Sikkim, South Sikkim 737139, India}
\date{}
\begin{document}
\parskip1ex
\parindent0pt
\maketitle

\begin{abstract}
    \noindent The Elliptic Somber and Euler Somber indices are newly defined topological indices based on the Somber index. Our paper presents calculations of the upper and lower bounds of these indices for the join and corona product of arbitrary graphs. Furthermore, we demonstrate that these bounds are attained when both graphs are regular.
\end{abstract}

\textbf{MSC2020 Classification:}

\section{Introduction}
    Topological indices based on vertex degrees are essential in mathematical chemistry and graph theory, providing quantitative insight into the structure of molecular graphs—these graphs model chemical compounds, with vertices and edges representing atoms and chemical bonds, respectively. The degree of a vertex indicates its connected edges, mirroring the atom's bond count. Such indices are invaluable for predicting a molecule's physical, chemical, and biological characteristics~\cite{harary1969graph}\cite{trinajstic2018chemical}.

    Introduced by Gutman in 2021, the Sombor index is a noteworthy topological measure derived from graph theory and Euclidean geometry. It has found extensive application in chemistry and pharmacology for its effectiveness in molecular structure characterization and biological activity prediction. This document explores the definition, attributes, and illustrative computation of the Sombor index~\cite{gutman2013degree}\cite{milovanovic2021some}.

    The Sombor index is a vertex degree-based index that adheres to the symmetry principle. It involves calculating the Euclidean distance from the origin $O$ to the degree-point $(d(u),d(v))$ of edge $uv$, represented as $\sqrt{d(u)^2+d(v)^2}$. This forms the basis of the Sombor index's definition~\cite{gutman2021geometric}.

    In 2023, Gutman et al.~\cite {gutman2024geometric} introduced the elliptic Sombor index, inspired by the elliptical orbits of planets, with the sun at one focus. This variant extends the traditional Sombor index by employing elliptical rather than Euclidean distances between degree points.

    Furthering this concept in 2024, Gutman's team demonstrated that the lengths of an ellipse's semi-major and semi-minor axes are equivalent. The perimeter approximation by Leonard Euler, $\pi \sqrt{(d(u)^2+d(v)^2)(d(u)+d(v))^2}$, led to the proposal of the Euler Sombor index as $\sqrt{d(u)^2+d(v)^2+d(u)d(v)}$. This presents a geometric analogy between the Sombor and Euler Sombor indices~\cite{ivan2024relating}.

    Research by Bibhas et al.~\cite{adhikari2023corona} delved into the structural nuances of corona graphs, examining signed links, signed triangles, and degree distributions. Their work also tackled the algebraic conflict in signed corona graphs derived from seed graphs. Further studies by Sheeba et al.~\cite{agnes2023index} provided explicit formulas for the Y-index in various corona graph products. Khalid et al.~\cite{khalid2022topological} investigated degree-based topological indices, including the Randić index and Zagreb indices, for bistar graphs and corona products. Iswadi et al.~\cite{iswadi2011metric} focused on the metric dimensions of corona products and graph joins. Dhananjaya et al.~\cite{kumara2021abc} derived the ABC and GA indices for several corona graph products. Arif et al.~\cite{arif2022sombor} calculated the Sombor index for various graph operations, including joins and corona products. Kirana et al.~\cite{kirana2024elliptic} computed the elliptic Sombor and Euler Sombor indices for graphs resulting from similar operations.

    This study aims to establish the bounds of the elliptic Sombor and Euler Sombor indices for graphs formed through join and Corona product operations, demonstrating that these bounds are achieved with regular graphs.

\section{Preliminaries}
    Let $G=(V(G), E(G))$ be a graph, where $V(G)$ is the set of vertices and $E$ is the set of edges. For a vertex $u\in V(G)$, the number of edges incident to $u$ is called its degree denoted by $d_G(u)$; if no confusion arises, then we can simply write $d(u)$. $\Delta_G$ is the highest vertex degree of the graph $G$ and $\delta_G$ is the lowest vertex degree of the graph $G$. So for any vertex $u\in V(G)$, $\delta_G\leq d(u_i)\leq\Delta_G$.

    \begin{definition}[Join]
        Let $G_1=(V_1,E_1)$ and $G_2=(V_2,E_2)$ be two graphs, the join $G_1+G_2$ is defined as: $$G_1+G_2=(V_1\cup V_2, E_1\cup E_2\cup \{(v_1,v_2):v_1\in V_1,v_2\in V_2\})$$.
    \end{definition}

    \begin{definition}[Corona Product]
        Let $G_1=(V_1,E_1)$ and $G_2=(V_2,E_2)$ be two graphs, the corona product of $G_1$ and $G_2$, is defined as a graph which is formed by taking $n$ copies of $G_2$ and connecting $i^{th}$ vertex of $G_1$ to the vertices of $G_2$ in each copies.
    \end{definition}

    Elliptic Somber and Euler Somber indices for a graph $G$ are defined as follows~\cite{gutman2024geometric}\cite{ivan2024relating}
    $$ESO(G)=\sum_{uv\in E(G)}(d(u)+d(v))\sqrt{d(u)^2+d(v)^2}$$
    $$EU(G)=\sum_{uv\in E(G)}\sqrt{d(u)^2+d(v)^2+d(u)d(v)}.$$

\section{Main Results}
    In this section, we calculate the bounds of the Elliptic Sombor index and Euler Sombor index of the join and corona product of arbitrary graphs $G_1$ and $G_2$.

    \begin{theorem}\label{T1}
        Let $G_1$ and $G_2$ be two graphs of order $n_1$ and $n_2$ and size $m_1$ and $m_2$ respectively. Then $\alpha_1 \leq ESO(G_1+G_2)\leq \alpha_2$, where\\
        $\alpha_1=2\sqrt{2}m_1(\delta_{G_1}+n_2)^2+2\sqrt{2}m_2(\delta_{G_2}+n_1)^2+n_1n_2(\delta_{G_1}+n_2+\delta_{G_2}+n_1)\{(\delta_{G_1}+n_2)^2+(\delta_{G_2}+n_1)^2\}^{\frac{1}{2}},$\\
        $\alpha_2=2\sqrt{2}m_1(\Delta_{G_1}+n_2)^2+2\sqrt{2}m_2(\Delta_{G_2}+n_1)^2+n_1n_2(\Delta_{G_1}+n_2+\Delta_{G_2}+n_1)\{(\Delta_{G_1}+n_2)^2+(\Delta_{G_2}+n_1)^2\}^{\frac{1}{2}}$.
    \end{theorem}
    \begin{proof}
        We know,
        \begin{equation*}
            \begin{split}
                ESO(G_1+G_2)&=\sum_{uv\in E(G_1+G_2)}(d_{G_1+G_2}(u)+d_{G_1+G_2}(v))\sqrt{d_{G_1+G_2}(u)^2+d_{G_1+G_2}(v)^2}\\
                &=\sum_{uv\in E(G_1)}(d(u)+n_2+d(v)+n_2)\sqrt{(d(u)+n_2)^2+(d(v)+n_2)^2}+\\&\sum_{uv\in E(G_2)}(d(u)+1+d(v)+1)\sqrt{(d(u)+1)^2+(d(v)+1)^2}+\\&\sum_{u\in V(G_1)}\sum_{v\in V(G_2)}(d(u)+n_2+d(v)+1)\sqrt{(d(u)+n_2)^2+(d(v)+1)^2}\\
                &\leq \sum_{uv\in E(G_1)}2(\Delta_{G_1}+n_2)\sqrt{2(\Delta_{G_1}+n_2)^2}+\sum_{uv\in E(G_2)}2(\Delta_{G_2}+1)\sqrt{2(\Delta_{G_2}+1)^2}+\\&\sum_{u\in V(G_1)}\sum_{v\in V(G_2)}(\Delta_{G_1}+n_2+\Delta_{G_2}+1)\sqrt{(\Delta_{G_1}+n_2)^2+(\Delta_{G_2}+1)^2}\\
                &=2\sqrt{2}m_1(\Delta_{G_1}+n_2)^2+2\sqrt{2}m_2(\Delta_{G_2}+n_1)^2+n_1n_2(\Delta_{G_1}+n_2+\Delta_{G_2}+n_1)\\&\{(\Delta_{G_1}+n_2)^2+(\Delta_{G_2}+n_1)^2\}^{\frac{1}{2}}.
            \end{split}
        \end{equation*}
        So, $ESO(G_1+G_2)\leq 2\sqrt{2}m_1(\Delta_{G_1}+n_2)^2+2\sqrt{2}m_2(\Delta_{G_2}+n_1)^2+n_1n_2(\Delta_{G_1}+n_2+\Delta_{G_2}+n_1)\{(\Delta_{G_1}+n_2)^2+(\Delta_{G_2}+n_1)^2\}^{\frac{1}{2}}.$\\
        Similarly, $ESO(G_1+G_2)\geq 2\sqrt{2}m_1(\delta_{G_1}+n_2)^2+2\sqrt{2}m_2(\delta_{G_2}+n_1)^2+n_1n_2(\delta_{G_1}+n_2+\delta_{G_2}+n_1)\{(\delta_{G_1}+n_2)^2+(\Delta_{G_2}+n_1)^2\}^{\frac{1}{2}}$.
    \end{proof}

    \begin{prop}
        Let $G_1$ and $G_2$ two $r_1$ and $r_2$ regular graphs. Then $ESO(G_1+G_2)=2\sqrt{2}m_1(r_1+n_2)^2+2\sqrt{2}m_2(r_2+n_1)^2+n_1n_2(r_1+r_2+n_1+n_2)$.
    \end{prop}
    \begin{proof}
        The proof follows directly from Theorem \ref{T1}.
    \end{proof}

    \begin{theorem}\label{T2}
        Let $G_1$ and $G_2$ be two graphs of order $n_1$ and $n_2$ and size $m_1$ and $m_2$ respectively. Then $\alpha_1 \leq EU(G_1+G_2)\leq \alpha_2$, where\\
        $\alpha_1=\sqrt{3}m_1(\delta_{G_1}+n_2)+\sqrt{3}n_1m_2(\delta_{G_2}+1)+n_1n_2\{(\delta_{G_1}+n_2)^2+(\delta_{G_2}+1)^2+(\delta_{G_1}+n_2)(\delta_{G_2}+1)\}^{\frac{1}{2}},$\\
        $\alpha_2=\sqrt{3}m_1(\Delta_{G_1}+n_2)+\sqrt{3}n_1m_2(\Delta_{G_2}+1)+n_1n_2\{(\Delta_{G_1}+n_2)^2+(\Delta_{G_2}+1)^2+(\Delta_{G_1}+n_2)(\Delta_{G_2}+1)\}^{\frac{1}{2}}$.
    \end{theorem}

    \begin{proof}
        We know,
        \begin{equation*}
            \begin{split}
                EU(G_1+G_2)&=\sum_{uv\in E(G_1+G_2)}\sqrt{d_{G_1+G_2}(u)^2+d_{G_1+G_2}(v)^2+d_{G_1+G_2}(u)d_{G_1+G_2}(v)}\\
                &=\sum_{uv\in E(G_1)}\sqrt{(d(u)+n_2)^2+(d(v)+n_2)^2+(d(u)+n_2)(d(v)+n_2)}+\\&\sum_{uv\in E(G_2)}\sqrt{(d(u)+n_1)^2+(d(v)+n_1)^2+(d(u)+n_1)(d(v)+n_1)}+\\&\sum_{u\in V(G_1)}\sum_{v\in V(G_2)}\sqrt{(d(u)+n_2)^2+(d(v)+n_1)^2+(d(u)+n_2)(d(v)+n_1)}\\
                &\leq \sum_{uv\in E(G_1)}\sqrt{(\Delta_{G_1}+n_2)^2+(\Delta_{G_1}+n_2)^2+(\Delta_{G_1}+n_2)(\Delta_{G_1}+n_2)}+\\&\sum_{uv\in E(G_2)}\sqrt{(\Delta_{G_2}+n_1)^2+(\Delta_{G_2}+n_1)^2+(\Delta_{G_2}+n_1)(\Delta_{G_2}+n_1)}+\\&\sum_{u\in V(G_1)}\sum_{v\in V(G_2)}\sqrt{(\Delta_{G_1}+n_2)^2+(\Delta_{G_2}+n_1)^2+(\Delta_{G_1}+n_2)(\Delta_{G_2}+n_1)}\\
                &=\sqrt{3}m_1(\Delta_{G_1}+n_2)+\sqrt{3}m_2(\Delta_{G_2}+n_1)+n_1n_2\{(\Delta_{G_1}+n_2)^2+(\Delta_{G_2}+n_1)^2+\\&(\Delta_{G_1}+n_2)(\Delta_{G_2}+n_1)\}^{\frac{1}{2}}.
            \end{split}
        \end{equation*}
        So, $EU(G_1+G_2)\leq \sqrt{3}m_1(\Delta_{G_1}+n_2)+\sqrt{3}m_2(\Delta_{G_2}+n_1)+n_1n_2\{(\Delta_{G_1}+n_2)^2+(\Delta_{G_2}+n_1)^2+(\Delta_{G_1}+n_2)(\Delta_{G_2}+1)\}^{\frac{1}{2}}.$\\
        Similarly, $EU(G_1+G_2)\geq \sqrt{3}m_1(\delta_{G_1}+n_2)+\sqrt{3}m_2(\delta_{G_2}+n_1)+n_1n_2\{(\delta_{G_1}+n_2)^2+(\delta_{G_2}+n_1)^2+(\delta_{G_1}+n_2)(\delta_{G_2}+1)\}^{\frac{1}{2}}.$

    \end{proof}

    \begin{prop}
        Let $G_1$ and $G_2$ two $r_1$ and $r_2$ regular graphs. Then $EU(G_1\circ G_2)=\sqrt{3}m_1(r_1+n_2)+\sqrt{3}m_2(r_2+n_1)+n_1n_2\{(r_1+n_2)^2+(r_2+n_1)^2+(r_1+n_2)(r_2+n_1)\}^{\frac{1}{2}}$.
    \end{prop}
    \begin{proof}
        The proof follows directly from Theorem \ref{T2}.
    \end{proof}

    \begin{theorem}\label{T3}
        Let $G_1$ and $G_2$ be two graphs of order $n_1$ and $n_2$ and size $m_1$ and $m_2$ respectively. Then $\alpha_1 \leq ESO(G_1\circ G_2)\leq \alpha_2$, where\\ $\alpha_1=2\sqrt{2}m_1(\delta_{G_1}+n_2)^2+2\sqrt{2}n_1m_2(\delta_{G_2}+1)^2+n_1n_2(\delta_{G_1}+n_2+\delta_{G_2}+1)\{(\delta_{G_1}+n_2)^2+(\delta_{G_2}+1)^2\}^{\frac{1}{2}},$\\
        $\alpha_2=2\sqrt{2}m_1(\Delta_{G_1}+n_2)^2+2\sqrt{2}n_1m_2(\Delta_{G_2}+1)^2+n_1n_2(\Delta_{G_1}+n_2+\Delta_{G_2}+1)\{(\Delta_{G_1}+n_2)^2+(\Delta_{G_2}+1)^2\}^{\frac{1}{2}}$.
    \end{theorem}

    \begin{proof}
        We know,
        \begin{equation*}
            \begin{split}
                ESO(G_1\circ G_2)&=\sum_{uv\in E(G_1\circ G_2)}(d_{G_1\circ G_2}(u)+d_{G_1\circ G_2}(v))\sqrt{d_{G_1\circ G_2}(u)^2+d_{G_1\circ G_2}(v)^2}\\
                &=\sum_{uv\in E(G_1)}(d(u)+n_2+d(v)+n_2)\sqrt{(d(u)+n_2)^2+(d(v)+n_2)^2}+\\&n_1\sum_{uv\in E(G_2)}(d(u)+1+d(v)+1)\sqrt{(d(u)+1)^2+(d(v)+1)^2}+\\&\sum_{u\in V(G_1)}\sum_{v\in V(G_2)}(d(u)+n_2+d(v)+1)\sqrt{(d(u)+n_2)^2+(d(v)+1)^2}\\
                &\leq \sum_{uv\in E(G_1)}2(\Delta_{G_1}+n_2)\sqrt{2(\Delta_{G_1}+n_2)^2}+n_1\sum_{uv\in E(G_2)}2(\Delta_{G_2}+1)\sqrt{2(\Delta_{G_2}+1)^2}+\\&\sum_{u\in V(G_1)}\sum_{v\in V(G_2)}(\Delta_{G_1}+n_2+\Delta_{G_2}+1)\sqrt{(\Delta_{G_1}+n_2)^2+(\Delta_{G_2}+1)^2}\\
                &=2\sqrt{2}m_1(\Delta_{G_1}+n_2)^2+2\sqrt{2}n_1m_2(\Delta_{G_2}+1)^2+n_1n_2(\Delta_{G_1}+n_2+\Delta_{G_2}+1)\\&\sqrt{(\Delta_{G_1}+n_2)^2+(\Delta_{G_2}+1)^2}.
            \end{split}
        \end{equation*}
        So, $ESO(G_1\circ G_2)\leq 2\sqrt{2}m_1(\Delta_{G_1}+n_2)^2+2\sqrt{2}n_1m_2(\Delta_{G_2}+1)^2+n_1n_2(\Delta_{G_1}+n_2+\Delta_{G_2}+1)\sqrt{(\Delta_{G_1}+n_2)^2+(\Delta_{G_2}+1)^2}.$\\
        Similarly, $ESO(G_1\circ G_2)\geq 2\sqrt{2}m_1(\delta_{G_1}+n_2)^2+2\sqrt{2}n_1m_2(\delta_{G_2}+1)^2+n_1n_2(\delta_{G_1}+n_2+\delta_{G_2}+1)\sqrt{(\delta_{G_1}+n_2)^2+(\delta_{G_2}+1)^2}.$

    \end{proof}

    \begin{prop}
        Let $G_1$ and $G_2$ two $r_1$ and $r_2$ regular graphs. Then $ESO(G_1\circ G_2)=2\sqrt{2}m_1(r_1+n_2)^2+2\sqrt{2}n_1m_2(r_2+1)^2+n_1n_2(r_1+r_2+1+n_2)$.
    \end{prop}
    \begin{proof}
        The proof follows directly from Theorem \ref{T3}.
    \end{proof}

    \begin{theorem}\label{T4}
        Let $G_1$ and $G_2$ be two graphs of order $n_1$ and $n_2$ and size $m_1$ and $m_2$ respectively. Then $\alpha_1 \leq EU(G_1\circ G_2)\leq \alpha_2$, where\\ $\alpha_1=\sqrt{3}m_1(\delta_{G_1}+n_2)+\sqrt{3}n_1m_2(\delta_{G_2}+1)+n_1n_2\{(\delta_{G_1}+n_2)^2+(\delta_{G_2}+1)^2+(\delta_{G_1}+n_2)(\delta_{G_2}+1)\}^{\frac{1}{2}},$\\
        $\alpha_2=\sqrt{3}m_1(\Delta_{G_1}+n_2)+\sqrt{3}n_1m_2(\Delta_{G_2}+1)+n_1n_2\{(\Delta_{G_1}+n_2)^2+(\Delta_{G_2}+1)^2+(\Delta_{G_1}+n_2)(\Delta_{G_2}+1)\}^{\frac{1}{2}}$.
    \end{theorem}

    \begin{proof}
        We know,
        \begin{equation*}
            \begin{split}
                EU(G_1\circ G_2)&=\sum_{uv\in E(G_1\circ G_2)}\sqrt{d_{G_1\circ G_2}(u)^2+d_{G_1\circ G_2}(v)^2+d_{G_1\circ G_2}(u)d_{G_1\circ G_2}(v)}\\
                &=\sum_{uv\in E(G_1)}\sqrt{(d(u)+n_2)^2+(d(v)+n_2)^2+(d(u)+n_2)(d(v)+n_2)}+\\&n_1\sum_{uv\in E(G_2)}\sqrt{(d(u)+1)^2+(d(v)+1)^2+(d(u)+1)(d(v)+1)}+\\&\sum_{u\in V(G_1)}\sum_{v\in V(G_2)}\sqrt{(d(u)+n_2)^2+(d(v)+1)^2+(d(u)+n_2)(d(v)+1)}\\
                &\leq \sum_{uv\in E(G_1)}\sqrt{(\Delta_{G_1}+n_2)^2+(\Delta_{G_1}+n_2)^2+(\Delta_{G_1}+n_2)(\Delta_{G_1}+n_2)}+\\&n_1\sum_{uv\in E(G_2)}\sqrt{(\Delta_{G_2}+1)^2+(\Delta_{G_2}+1)^2+(\Delta_{G_2}+1)(\Delta_{G_2}+1)}+\\&\sum_{u\in V(G_1)}\sum_{v\in V(G_2)}\sqrt{(\Delta_{G_1}+n_2)^2+(\Delta_{G_2}+1)^2+(\Delta_{G_1}+n_2)(\Delta_{G_2}+1)}\\
                &=\sqrt{3}m_1(\Delta_{G_1}+n_2)+\sqrt{3}n_1m_2(\Delta_{G_2}+1)+n_1n_2\{(\Delta_{G_1}+n_2)^2+(\Delta_{G_2}+1)^2+\\&(\Delta_{G_1}+n_2)(\Delta_{G_2}+1)\}^{\frac{1}{2}}.
            \end{split}
        \end{equation*}
        So, $EU(G_1\circ G_2)\leq \sqrt{3}m_1(\Delta_{G_1}+n_2)+\sqrt{3}n_1m_2(\Delta_{G_2}+1)+n_1n_2\{(\Delta_{G_1}+n_2)^2+(\Delta_{G_2}+1)^2+(\Delta_{G_1}+n_2)(\Delta_{G_2}+1)\}^{\frac{1}{2}}.$\\
        Similarly, $EU(G_1\circ G_2)\geq \sqrt{3}m_1(\delta_{G_1}+n_2)+\sqrt{3}n_1m_2(\delta_{G_2}+1)+n_1n_2\{(\delta_{G_1}+n_2)^2+(\delta_{G_2}+1)^2+(\delta_{G_1}+n_2)(\delta_{G_2}+1)\}^{\frac{1}{2}}.$

    \end{proof}

    \begin{prop}
        Let $G_1$ and $G_2$ two $r_1$ and $r_2$ regular graphs. Then $EU(G_1\circ G_2)=\sqrt{3}m_1(r_1+n_2)+\sqrt{3}n_1m_2(r_2+1)+n_1n_2\{(r_1+n_2)^2+(r_2+1)^2+(r_1+n_2)(r_2+1)\}^{\frac{1}{2}}$.
    \end{prop}
    \begin{proof}
        The proof follows directly from Theorem \ref{T4}.
    \end{proof}

\bibliographystyle{abbrv} 
\bibliography{main.bib}

\end{document}